\documentclass[runningheads,envcountsame]{llncs}
\usepackage{graphicx}
\usepackage[T1]{fontenc}
\usepackage{latexsym}
\usepackage{amsmath}
\usepackage{amssymb}
\usepackage{mleftright}
\usepackage{pgf}
\usepackage{tikz}
\usepackage{url}
\usetikzlibrary{scopes}
\usetikzlibrary{arrows,automata}
\definecolor{zzttqq}{rgb}{0.6,0.2,0}

\renewcommand{\vec}[1]{{\boldsymbol{#1}}}
\newcommand{\set}[1]{\ensuremath{\mleft\{ #1 \mright\}}}
\newcommand{\apply}[1]{\ensuremath{\mleft( #1 \mright)}}
\newcommand{\tup}[2]{\mleft(#1_1,\ldots,#1_{#2}\mright)}
\newcommand{\tupone}[2]{(#1,#1,\ldots,\stackrel{#2}{#1})}
\newcommand{\tupdos}[4]{(#1,#1,\ldots,\stackrel{#2}{#1},#3,#3,\ldots,\stackrel{#4}{#3})}
\newcommand{\tuptres}[6]{(#1,\ldots,\stackrel{#2}{#1},#3,\ldots,\stackrel{#4}{#3},#5,\ldots,\stackrel{#6}{#5})}
\newcommand{\Lp}[1]{\ensuremath{\mathcal{L}_{#1}}}
\DeclareMathOperator{\Part}{Part}
\DeclareMathOperator{\Cl}{C}
\DeclareMathOperator{\Ss}{S}
\DeclareMathOperator{\PP}{P}
\DeclareMathOperator{\NP}{NP}
\DeclareMathOperator{\NS}{NS}
\DeclareMathOperator{\lo}{o}

\newcommand{\Na}{\mathbb{N}}
\newcommand{\Nplus}{\mathbb{N}_{+}}
\newcommand{\Pa}{\mathbb{P}}
\newcommand{\Qa}{\mathbb{Q}}
\newcommand{\MMM}{\mathbb{M}_{3}}
\newcommand{\NNNNN}{\mathbb{N}_{5}}

\begin{document}
\def\doi#1{\url{https://doi.org/#1}}%
\title{Representing partition lattices through FCA\thanks{The third author gratefully acknowledges financial support by the Asociaci\'on Mexicana de Cultura A.C. She also would like to thank Prof.\ Bernhard Ganter for pointing out the topic and Dr. Christian Meschke for his constant support.}}
%
%
\author{Mike Behrisch\inst{1}\orcidID{0000-0003-0050-8085}\and Alain Chavarri Villarello\inst{2}\and\\
Edith Vargas-Garc\'\i a\inst{2}\orcidID{0000-0001-9677-9087}}
%
\authorrunning{M.~Behrisch\and A.\,Chavarri Villarello\and E.~Vargas-Garc\'\i{}a}
%
\institute{Institute of Discrete Mathematics and Geometry, Technische Universit\"at Wien\\
\email{behrisch@logic.at}
\and Department of Mathematics, ITAM R\'\i o Hondo, Ciudad de M\'exico\\
\email{\{achavar3,edith.vargas\}@itam.mx}}
\maketitle              
\begin{abstract}
We investigate the standard context, denoted by $\mathbb{K}\left(\mathcal{L}_{n}\right)$, of the lattice~$\mathcal{L}_{n}$ of partitions of a positive  integer~$n$ under the dominance order. Motivated by the discrete dynamical model to study integer partitions by Latapy and Duong Phan and by the characterization of the supremum and (infimum) irreducible partitions of~$n$ by Brylawski, we show how to construct the join-irreducible elements of $\mathcal{L}_{n+1}$ from~$\mathcal{L}_{n}$.  We employ this construction to count the number of join-irreducible elements of~$\mathcal{L}_{n}$, and show that the number of objects (and attributes) of $\mathbb{K}\left(\mathcal{L}_{n}\right)$ has order $\Theta(n^2)$. 

\end{abstract}
\section{Introduction}\label{sect:intro}
The study of \emph{partitions} of an integer started to gain attention in 1674 when Leibniz investigated~\cite[p.~37]{mahnke1912leibniz} the number of ways one can write a positive integer~$n$ as a sum of positive integers in decreasing order, which he called `divulsiones', today known as (unrestricted) partitions, see~\cite{OEISA000041}. He observed that there are~$5$ partitions of the number~$4$, namely the partitions $4, 3+1, 2+2, 2+1+1$ and $1+1+1+1$; for the number~$5$ there are~$7$ partitions, for~$6$ there are~$11$ partitions etc., and so he asked about the number~$p(n)$ of partitions of a positive integer~$n$. In 1918 G.\,H.~Hardy and S.~Ramanujan in~\cite{hardy1918} published an asymptotic formula to count~$p(n)$. To our knowledge, until now, there is `no closed-form expression' known to express~$p(n)$ for any integer~$n$.

From~\cite{Brylawski}, it is known that the set of all partitions of a positive integer~$n$ with \emph{the dominance ordering} (defined in the next section)  is a \emph{lattice}, denoted by~$\Lp{n}$. Moreover, Brylawski proposed a dynamical approach to study this lattice, which is explained in a more intuitive form by Latapy and Duong Phan in~\cite{LATAPY}. Motivated by their method to construct $\Lp{n+1}$ from~$\Lp{n}$, we restrict their approach  to the join-irreducible elements of~$\Lp{n}$ and show how to construct the join-irreducible elements of $\Lp{n+1}$ from those of~$\Lp{n}$. Our second main result is to give a recursive formula for the number of join-irreducible elements of~$\Lp{n}$, and since partition conjugation is a lattice antiautomorphism making~$\Lp{n}$ autodual (see~\cite{Brylawski}), we also derive the number of meet-irreducible elements. Finally we prove that the number of objects (and attributes) of $\mathbb{K}\apply{\Lp{n}}$ has order $\Theta(n^2)$, and we show how to obtain the standard context $\mathbb{K}\apply{\Lp{n+1}}$ from $\mathbb{K}\apply{\Lp{n}}$, giving a near-optimal algorithm for~$\mathbb{K}(\Lp{n})$ of time complexity~$\Theta(n^4)$.

The sections of this paper should be read in consecutive order. The following one introduces some notation and prepares basic definitions and facts concerning lattice theory and formal concept analysis. In Section~\ref{sect:Ln-Ln+1} we explore the relation between the join-irreducible elements of the lattice $\Lp{n+1}$ and those of the lattice~$\Lp{n}$, and the final section concludes the task by counting the supremum irreducible elements of~$\Lp{n}$.

\section{Preliminaries}\label{sect:preliminaries}
\subsection{Lattices and partitions}
Throughout the text $\Na:=\set{0,1,2,\ldots}$ denotes the set of \emph{natural numbers} and $\Nplus:=\set{1,2,\ldots}$ denotes the set of \emph{positive integers}. Moreover, if~$P$ is a set and $\leq$ is a binary relation on~$P$, which is \emph{reflexive}, \emph{antisymmetric} and \emph{transitive},  then $\Pa:=\apply{P,\leq}$ is \emph{a partially ordered set}. We often identify~$\Pa$ and~$P$; for example, we write $x\in\Pa$ instead of~$x\in P$ when this is convenient. For such a partially ordered set~$\Pa$ and elements $a,b\in \Pa$, we say that~$a$ is \emph{covered} by~$b$ (or~$b$ \emph{covers}~$a$), and write $a \prec b$, if $a<b$ and $a\leq z <b$ implies $z=a$. Partially ordered sets~$\Pa$ and~$\Qa$ are \emph{(order-)isomorphic} if there is a bijective map $\varphi\colon\Pa \to \Qa$ such that $a\leq b$ in~$\Pa$ if and only if $\varphi(a)\leq \varphi(b)$ in~$\Qa$. Every such~$\varphi$ is called an \emph{order-isomorphism}.

  A \emph{lattice} is a partially ordered set~$\Pa$ such that any two elements~$a$ and~$b$ have a least upper bound, called \emph{supremum} of~$a$ and~$b$, denoted by $a\vee b$, and a greatest lower bound, called \emph{infimum} of~$a$ and~$b$, denoted by $a \wedge b$. Moreover, if supremum $\bigvee S$ and infimum $\bigwedge S$ exist for all $S\subseteq \Pa$, then~$\Pa$ is called a \emph{complete lattice}. Let~$\mathbb{L}$ be a lattice and $M \subseteq \mathbb{L}$. Then~$M$ is a (carrier set of a) \emph{sublattice}~$\mathbb{M}$ of~$\mathbb{L}$ if whenever $a,b\in M$, then $a\vee b \in M$ and $a\wedge b \in M$.

For a lattice~$\mathbb{L}$, we say that $a\in \mathbb{L}$ is \emph{join-irreducible}, denoted by $\vee$-irreducible,  if it is not \emph{the minimum element} $0_{\mathbb{L}}$ (provided such an element exists) and for every $b,c\in \mathbb{L}$ such that $a=b\vee c$ we can conclude
that $a=b$ or $a=c$. In particular for a finite lattice, the join-irreducible elements are those which cover precisely one element. Similarly, we say that $a\in \mathbb{L}$ is \emph{meet-irreducible}, denoted by $\wedge$-irreducible,  if it is not \emph{the maximum element} $1_{\mathbb{L}}$ (provided such an element exists) and for every $b,c\in \mathbb{L}$ such that $a=b\wedge c$
we can conclude that $a=b$ or $a=c$. In particular for a finite lattice, the meet-irreducible elements are those which are covered by precisely one element. For a lattice~$\mathbb{L}$,  we denote by $\mathcal{J}\apply{\mathbb{L}}$ and by $\mathcal{M}\apply{\mathbb{L}}$ the set of \emph{all join-irreducible elements}, and of \emph{all meet-irreducible elements} of~$\mathbb{L}$, respectively. As an example, the meet-irreducible elements of the lattices~$\NNNNN$ and $\MMM$ appear shaded in Figure~\ref{Fig.1}.

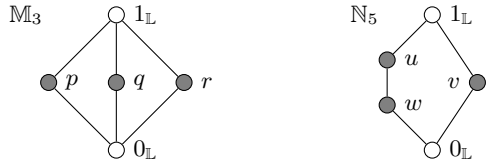
\begin{figure}[h]
\centering
			\tikzstyle{normal}=[circle,draw,inner sep=.075cm]
			\begin{tikzpicture}[scale=0.60]
			
			\node (11) at (-1.5,-4) {$\NNNNN$};
			\node[normal,label=right:{$1_{\mathbb{L}}$}] (1) at (0, -4) {};
			\node[normal,label=right:{$u$},fill=gray] (2) at (-1, -5) {};
			\node[normal,label=right:{$w$},fill=gray] (3) at (-1, -6) {};
			\node[normal,label=left:{$v$},fill=gray] (4) at (1, -5.5) {};
			\node[normal,label=right:{$0_{\mathbb{L}}$}] (5) at (0, -7) {};
			\draw (1) -- (2) --(3) --(5);
			\draw (1) -- (4) --(5);
			
			\node (11) at (-9,-4) {$\MMM$};
			\node[normal,label=right:{$1_\mathbb{L}$}] (6) at (-7, -4) {};
			\node[normal,label=right:{$p$},fill=gray] (7) at (-8.5, -5.5) {};
			\node[normal,label=right:{$q$},fill=gray] (8) at (-7, -5.5) {};
			\node[normal,label=right:{$r$},fill=gray] (9) at (-5.5, -5.5) {};
			\node[normal,label=right:{$0_\mathbb{L}$}] (10) at (-7, -7) {};
			\draw (6) -- (7) --(10);
			\draw (6) -- (8) --(10);
			\draw (6) -- (9) --(10);
			
			\end{tikzpicture}

\caption{Lattices $\MMM$ and $\NNNNN$ with meet-irreducible elements appearing shaded.}
\label{Fig.1}			
\end{figure}			

Our aim is to study the join-irreducible elements in the lattice of positive integer partitions. We thus define a partition formally.
\begin{definition}
An (ordered) \emph{partition} of a positive integer $n\in \Nplus$ is an~$n$-tuple $\vec{\alpha}:=\tup{a}{n}$ of natural numbers such that
\begin{align*}
a_1&\geq a_2\geq \ldots \geq a_n\geq 0& \text{and}&& n&=a_1+a_2+\ldots+a_n.
\end{align*}
If there is $k\in \set{1,\ldots,n}$ such that $a_k> 0$ and $a_i=0$ for all $i>k$, the partition $\vec{\alpha}$ can be written in the form $\apply{a_1, \ldots, a_k}$, whereby we delete the zeros at the end.
\end{definition}

For example, $(3,2,2,1,0,0,0,0)$ is a partition of~$8$ because ${3\geq 2\geq 2\geq 1\geq 0}$ and $3+2+2+1=8$. By deleting the zeros, we write more succinctly $(3,2,2,1)$ for this partition. Graphically, we can illustrate a partition using a diagram that has a ladder shape, which is known as a Ferrers diagram (cf.\ Figure~\ref{FigYoungDiagram}). Given a partition~$\vec{\alpha}$ of~$n$, it is possible to obtain \emph{the conjugated partition}~$\vec{\alpha}^*$ in the sense of~\cite{Brylawski} by reading its Ferrers diagram by rows, from bottom to top. This operation can also be seen as reflecting the Ferrers diagram along a diagonal axis. For instance, the partition $\apply{3,2,2,1}$ of $8=3+2+2+1$ has the Ferrers diagram presented in Figure~\ref{FigYoungDiagram}, and its conjugate consists of $4$~grains from the first row, $3$~grains from the second, and $1$~grain from the third. So we get the partition $\apply{4,3,1}$; its Ferrers diagram is also shown in Figure~\ref{FigYoungDiagram} below.
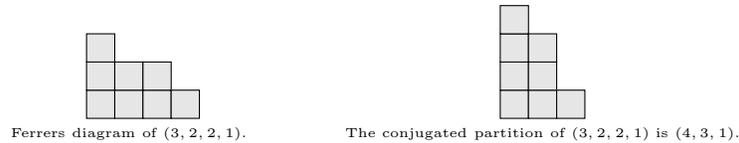
\begin{figure}[h]
\centering
			
			\tikzstyle{normal}=[rectangle,draw,inner sep=0.1875cm,fill=gray!20]
			
			\begin{tikzpicture}[x=.75cm,y=.75cm]
			\node[normal] (1) at (0, 0) {};
			\node[normal] (2) at (0, 0.5) {};
			\node[normal] (1) at (0, 1) {};
			\node[normal,label=below:\tiny{Ferrers diagram of $\apply{3,2,2,1}$.}] (1) at (0.5, 0) {};
			\node[normal] (1) at (0.5, 0.5) {};
			\node[normal] (1) at (1, 0) {};
			\node[normal] (1) at (1, 0.5) {};
			\node[normal] (1) at (1.5, 0) {};
			
			\begin{scope}[xshift=5.5cm]
			\node[normal] (1) at (0, 0) {};
			\node[normal] (2) at (0, 0.5) {};
			\node[normal] (1) at (0, 1) {};
			\node[normal] (1) at (0, 1.5) {};
			\node[normal,label=below:\tiny{The conjugated partition of $\apply{3,2,2,1}$ is $\apply{4,3,1}$.}] (1) at (0.5, 0) {};
			\node[normal] (1) at (0.5, 0.5) {};
			\node[normal] (1) at (0.5, 1) {};
			\node[normal] (1) at (1, 0) {};
			\end{scope}
			
			\end{tikzpicture}
			
\caption{Ferrers diagrams}
\label{FigYoungDiagram}
			
\end{figure}

The \emph{set of all partitions} of $n\in \Nplus$, denoted by $\Part(n)$, can be ordered in different ways. One of them is the dominance ordering, defined as follows.

\begin{definition}[\cite{Brylawski}]
Let $\vec{\alpha}=\tup{a}{k}$ and $\vec{\beta}=\tup{b}{m}$ in $\Part(n)$ be two partitions of $n\in \Nplus$. We define \emph{the dominance ordering} between $\vec{\alpha}$ and $\vec{\beta}$ by
\[\vec{\alpha}\geq \vec{\beta} \quad\text{if and only if}\quad \sum_{i=1}^j a_i \geq \sum_{i=1}^j b_i \text{ for all } j\geq 1. \]
\end{definition}

In~\cite[Proposition 2.2]{Brylawski} it is shown that the set $\Part(n)$ with the dominance ordering is 
a lattice. 
We denote by~$\Lp{n} =\apply{\Part(n),\leq}$ \emph{the lattice of all partitions} of $n \in \Nplus$ with the dominance ordering. Characterizing the covering relation is central for the construction of a finite lattice, and the following theorem provides this characterization in the case of~$\Lp{n}$.

\begin{theorem}[Brylawski~\cite{Brylawski}]\label{covering}
In the lattice~$\Lp{n}$, the partition $\vec{\alpha}=(a_1,\ldots,a_n)$ covers  $\vec{\beta}=(b_1,\ldots,b_n)$, denoted $\vec{\beta} \prec \vec{\alpha}$, if and only if either of the following two cases (not necessarily disjoint) is satisfied:
		\begin{enumerate}
			\item There exists $j\in \set{1,\ldots,n}$ such that $a_j=b_j+1$, $a_{j+1}=b_{j+1}-1$ and $a_i = b_i $ for all $i\neq j$ and $i\neq j+1$.
			\item There exist $j,k\in \set{1,\ldots,n}$ such that $a_j=b_j+1$, $a_k=b_k-1$, $b_j=b_k$ and $a_i=b_i$ for all $i\neq j$ and $i\neq k$.
		\end{enumerate}
\end{theorem}

In~\cite{LATAPY} Matthieu Latapy and Ha Duong Phan give a dynamic approach to study the lattice~$\Lp{n}$ so that the dominance ordering and the covering relation can be visualized more intuitively. To this end they use the following definitions.
\begin{definition}[\cite{LATAPY}]
For a partition $\vec{\alpha}=\tup{a}{k}$, the \emph{height difference} of $\vec{\alpha}$ at~$j$, denoted by $d_j\apply{\vec{\alpha}}$, is the integer $a_j-a_{j+1}$. We say that the partition $\vec{\alpha}=\tup{a}{k}$ has:
\begin{enumerate}
 \item a \emph{cliff} at~$j$ if $d_j\apply{\vec{\alpha}}\geq 2$.
 \item a \emph{slippery plateau} at~$j$ if there exists $k>j$ such that $d_i\apply{\vec{\alpha}}=0$ for all $i \in \set{j,j+1,\ldots, k-1}$ and $d_k\apply{\vec{\alpha}}=1$. The integer $k-j$ is called the \emph{length} of the slippery plateau at~$j$.
 \item a \emph{non-slippery plateau} at~$j$ if $d_i\apply{\vec{\alpha}}=0$ for all $i \in \set{j,j+1,\ldots, k-1}$ and it has a cliff at~$k$. The integer $k-j$ is the \emph{length} of the non-slippery plateau at~$j$.
 \item a \emph{slippery step} at~$j$ if $\vec{\alpha}'=\apply{a_1,\ldots,a_j-1,\ldots,a_k}$ is a partition with a slippery plateau at~$j$.
 \item a \emph{non-slippery step} at~$j$ if $\vec{\alpha}'=\apply{a_1,\ldots,a_j-1,\ldots,a_k}$  is a partition with a non-slippery plateau at~$j$.
\end{enumerate}

\end{definition}

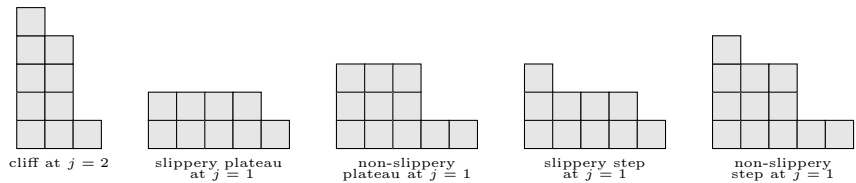
\begin{figure}[h]
\centering		
				\begin{tikzpicture}[x=.75cm, y=.75cm]
				\tikzstyle{normal}=[rectangle,draw,inner sep=0.1875cm,fill=gray!20]
				\node[normal] (1) at (-3, 0) {};
				\node[normal] (2) at (-3, 0.5) {};
				\node[normal] (3) at (-3, 1) {};
				\node[normal] (4) at (-3, 1.5) {};
				\node[normal] (5) at (-3, 2) {};
				\node[normal,label=below:\tiny{cliff at $j=2$}] (1) at (-2.5, 0) {};
				\node[normal] (6) at (-2.5, 0.5) {};
				\node[normal] (7) at (-2.5, 1) {};
				\node[normal] (8) at (-2.5, 1.5) {};
				\node[normal] (9) at (-2, 0) {};
				
				\begin{scope}[xshift=-0.5cm]
				
				\node[normal] (1) at (0, 0) {};
				\node[normal] (2) at (0, 0.5) {};
				\node[normal] (3) at (0.5, 0) {};
				\node[normal] (4) at (0.5, 0.5) {};
				\node[normal,label=below:\tiny{slippery plateau}] (1) at (1, 0) {};
				\node[label=below:\tiny{ at $j=1$}] (1) at (1, -0.3) {};
				\node[normal] (5) at (1, 0.5) {};
				\node[normal] (6) at (1.5, 0) {};
				\node[normal] (7) at (1.5, 0.5) {};
				\node[normal] (8) at (2, 0) {};
				
				\end{scope}
				\begin{scope}[xshift=2.cm]
				
				\node[normal] (1) at (0, 0) {};
				\node[normal] (2) at (0, 0.5) {};
				\node[normal] (3) at (0, 1) {};
				\node[normal] (4) at (0.5, 0) {};
				\node[normal] (5) at (0.5, 0.5) {};
				\node[normal] (6) at (0.5, 1) {};
				\node[normal,label=below:\tiny{non-slippery}] (1) at (1, 0) {};
				\node[label=below:\tiny{plateau at $j=1$}] (1) at (1, -0.3) {};
				\node[normal] (7) at (1, 0.5) {};
				\node[normal] (8) at (1, 1) {};
				\node[normal] (9) at (1.5, 0) {};
				\node[normal] (10) at (2, 0) {};
				
				\end{scope}
				
				\begin{scope}[xshift=4.5cm]
				
				\node[normal] (1) at (0, 0) {};
				\node[normal] (1) at (0, 0.5) {};
				\node[normal] (1) at (0, 1) {};
				\node[normal] (1) at (0.5, 0) {};
				\node[normal] (1) at (0.5, 0.5) {};
				\node[normal,label=below:\tiny{slippery step}] (1) at (1, 0) {};
				\node[label=below:\tiny{at $j=1$}] (1) at (1, -0.3) {};
				\node[normal] (1) at (1, 0.5) {};
				\node[normal] (1) at (1.5, 0) {};
				\node[normal] (1) at (1.5, 0.5) {};
				\node[normal] (1) at (2, 0) {};

				\end{scope}
				
				\begin{scope}[xshift=7cm]
				\node[normal] (1) at (0, 0) {};
				\node[normal] (1) at (0, 0.5) {};
				\node[normal] (1) at (0, 1) {};
				\node[normal] (1) at (0, 1.5) {};
				\node[normal] (1) at (0.5, 0) {};
				\node[normal] (1) at (0.5, 0.5) {};
				\node[normal] (1) at (0.5, 1) {};
				\node[normal,label=below:\tiny{non-slippery}] (1) at (1, 0) {};
				\node[label=below:\tiny{step at $j=1$}] (1) at (1, -0.3) {};
				\node[normal] (1) at (1, 0.5) {};
				\node[normal] (1) at (1, 1) {};
				\node[normal] (1) at (1.5, 0) {};
				\node[normal] (1) at (2, 0) {};
				\end{scope}

				\end{tikzpicture}
				
\caption{The slippery plateau and the slippery step have length~$3$, the non-slippery plateau and the non-slippery step have length~$2$.}
\end{figure}

The covering relation of Theorem~\ref{covering} is also described in~\cite{LATAPY} using a correspondence between the two cases of Theorem~\ref{covering} and the following \emph{two transition rules}.

\begin{enumerate}
 \item If $\vec{\alpha}$ has a cliff at~$j$, one grain can fall from column~$j$ to column $j+1$.

		
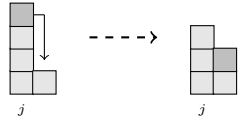
\begin{center}	
			\begin{tikzpicture}[scale=0.6]
			\tikzstyle{normal}=[rectangle,draw,inner sep=0.1555cm,fill=gray!20]
			\node[normal,label=below:\tiny{$j$}] (1) at (0, 0) {};
			\node[normal] (1) at (0, 0.5) {};
			\node[normal] (2) at (0, 1) {};
			\node[normal,fill=gray!50] (3) at (0, 1.5) {};
			\node[normal] (4) at (0.5, 0) {};
			
			\draw [->,thick,dashed] (1.5,1)--(3,1);
			\draw [-] (0.25,1.5)--(0.5,1.5);
			\draw [->] (0.5,1.5)--(0.5,0.5);
			
			\node[normal,label=below:\tiny{$j$}] (5) at (4, 0) {};
			\node[normal] (6) at (4, 0.5) {};
			\node[normal] (7) at (4, 1) {};
			\node[normal] (8) at (4.5, 0) {};
			\node[normal,fill=gray!50] (9) at (4.5, 0.5) {};
			
			\end{tikzpicture}
			
		\end{center}
		
 \item If $\vec{\alpha}$ has a slippery step of length~$l$ at~$j$, one grain can slip from column~$j$ to column $k=j+l+1$.
        
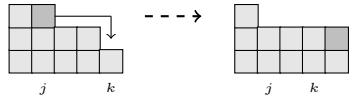
\begin{center}			
			\begin{tikzpicture}[scale=0.6]
			\tikzstyle{normal}=[rectangle,draw,inner sep=0.1555cm,fill=gray!20]
			\node[normal] (1) at (0, 0) {};
			\node[normal] (2) at (0, 0.5) {};
			\node[normal] (3) at (0, 1) {};
			\node[normal,label=below:\tiny{$j$}] (4) at (0.5, 0) {};
			\node[normal] (5) at (0.5, 0.5) {};
			\node[normal,fill=gray!50] (6) at (0.5, 1) {};
			\node[normal] (7) at (1, 0) {};
			\node[normal] (8) at (1, 0.5) {};
			\node[normal] (9) at (1.5, 0) {};
			\node[normal] (10) at (1.5, 0.5) {};
			\node[normal,label=below:\tiny{$k$}] (11) at (2, 0) {};

			\draw [->,thick,dashed] (2.75,1)--(4,1);
			\draw [-] (0.75,1)--(2,1);
			\draw [->] (2,1)--(2,0.5);
			\node[normal] (12) at (5, 0) {};
			\node[normal] (13) at (5, 0.5) {};
			\node[normal] (14) at (5, 1) {};
			\node[normal,label=below:\tiny{$j$}] (15) at (5.5, 0) {};
			\node[normal] (16) at (5.5, 0.5) {};
			\node[normal] (17) at (6, 0) {};
			\node[normal] (18) at (6, 0.5) {};
			\node[normal,label=below:\tiny{$k$}] (19) at (6.5, 0) {};
			\node[normal] (20) at (6.5, 0.5) {};
			\node[normal] (21) at (7, 0) {};
			\node[normal,fill=gray!50] (22) at (7, 0.5) {};
			\end{tikzpicture}
			
		\end{center}
\end{enumerate}

Given a partition $\vec{\alpha}\in \Lp{n}$, if one applies to $\vec{\alpha}$ the first or second transition rule, every time that~$\vec{\alpha}$ has a cliff or slippery step, then we obtain all partitions~$\vec{\beta}$ such that $\vec{\beta}\prec \vec{\alpha}$. If~$\vec{\beta}$ is obtained from~$\vec{\alpha}$ by applying one of the two transition rules, then~$\vec{\beta}$ is called \emph{directly reachable} from~$\vec{\alpha}$. This is denoted by $\vec{\alpha} \xrightarrow[]{j} \vec{\beta}$, where~$j$ is the column from which the grain falls or slips. The set ${\mathcal{D}(\vec{\alpha})=\{\vec{\beta}\in\Lp{n} \mid \exists j\colon \vec{\alpha}  \xrightarrow[]{j} \vec{\beta} \}=\set{\vec{\beta}\in\Lp{n} \mid \vec{\beta} \prec \vec{\alpha}}}$ denotes the set of \textit{all partitions directly reachable} from $\vec{\alpha}$.

	\begin{example}
 Consider the partition $\vec{\alpha}=\apply{5,3,2,1,1}$ in~$\Lp{12}$.
 It has a cliff at $j=1$, and also has a slippery step at $j=2$ and at $j=3$. We obtain $\vec{\beta}_1$, $\vec{\beta}_2$ and $\vec{\beta}_3$ after applying the corresponding transition rules $\vec{\alpha} \xrightarrow[]{1} \vec{\beta}_1$, $\vec{\alpha} \xrightarrow[]{2} \vec{\beta}_2$, $\vec{\alpha} \xrightarrow[]{3} \vec{\beta}_3$, the Ferrers diagrams of which are:
		
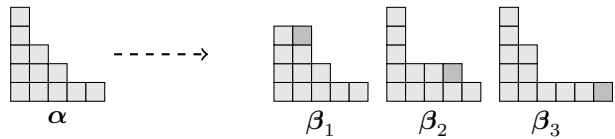
\begin{center}
			\begin{tikzpicture}[x=1cm, y=1cm,scale=0.5]
			\tikzstyle{normal}=[rectangle,draw,inner sep=0.125cm,fill=gray!20]
			
			\node[normal] (1) at (0, 0) {};
			\node[normal] (2) at (0, 0.5) {};
			\node[normal] (3) at (0, 1) {};
			\node[normal] (4) at (0, 1.5) {};
			\node[normal] (5) at (0, 2) {};
			\node[normal] (6) at (0.5, 0) {};
			\node[normal] (7) at (0.5, 0.5) {};
			\node[normal] (8) at (0.5, 1) {};
			\node[normal,label=below:{$\vec{\alpha}$}] (9) at (1, 0) {};
			\node[normal] (10) at (1, 0.5) {};
			\node[normal] (11) at (1.5, 0) {};
			\node[normal] (12) at (2, 0) {};
			
			\draw [->,thick,dashed] (2.5,1)--(5,1);
			
			\begin{scope}[xshift=7cm]
			\node[normal] (1) at (0, 0) {};
			\node[normal] (1) at (0, 0.5) {};
			\node[normal] (1) at (0, 1) {};
			\node[normal] (1) at (0, 1.5) {};
			\node[normal] (1) at (0.5, 0) {};
			\node[normal] (1) at (0.5, 0.5) {};
			\node[normal] (1) at (0.5, 1) {};
			\node[normal,fill=gray!50] (1) at (0.5, 1.5) {};
			\node[normal,label=below:{$\vec{\beta}_1$}] (1) at (1, 0) {};
			\node[normal] (1) at (1, 0.5) {};
			\node[normal] (1) at (1.5, 0) {};
			\node[normal] (1) at (2, 0) {};
			
			\node[normal] (1) at (3, 0) {};
			\node[normal] (1) at (3, 0.5) {};
			\node[normal] (1) at (3, 1) {};
			\node[normal] (1) at (3, 1.5) {};
			\node[normal] (1) at (3, 2) {};
			\node[normal] (1) at (3.5, 0) {};
			\node[normal] (1) at (3.5, 0.5) {};
			\node[normal,label=below:{$\vec{\beta}_2$}] (1) at (4, 0) {};
			\node[normal] (1) at (4, 0.5) {};
			\node[normal] (1) at (4.5, 0) {};
			\node[normal,fill=gray!50] (1) at (4.5, 0.5) {};
			\node[normal] (1) at (5, 0) {};

			\node[normal] (1) at (6, 0) {};
			\node[normal] (1) at (6, 0.5) {};
			\node[normal] (1) at (6, 1) {};
			\node[normal] (1) at (6, 1.5) {};
			\node[normal] (1) at (6, 2) {};
			\node[normal] (1) at (6.5, 0) {};
			\node[normal] (1) at (6.5, 0.5) {};
			\node[normal] (1) at (6.5, 1) {};
			\node[normal,label=below:{$\vec{\beta}_3$}] (1) at (7, 0) {};
			\node[normal] (1) at (7.5, 0) {};
			\node[normal] (1) at (8, 0) {};
			\node[normal,fill=gray!50] (1) at (8.5, 0) {};
			\end{scope}
			
			\end{tikzpicture}	
		\end{center}
 Therefore, $\mathcal{D}\apply{\vec{\alpha}}=\set{(4,4,2,1,1),(5,2,2,2,1),(5,3,1,1,1,1)}$.
	\end{example}
	
We would like to study the standard context of the lattice~$\Lp{n}$. To this end, we quickly recall the notion of context and standard context as known from formal concept analysis~\cite{ganter}.

\subsection{Notions of formal concept analysis}

Formal concept analysis (FCA) is a method for data analysis based on the notion of a \emph{formal concept}. FCA was born around 1980, when a research group in Darmstadt, Germany, headed by Rudolf Wille, began to develop a framework for lattice theory applications, \cite{ganter1997applied}.

For sets~$G$ and~$M$ and any binary relation $I\subseteq G\times M$ between~$G$ and~$M$ the triple $\apply{G,M,I}$ is called a \emph{formal context}, and~$I$ its \emph{incidence relation}. The elements of~$G$ are called \emph{objects} and those of~$M$ are \emph{attributes}. When~$G$ and~$M$ are finite, a  formal context can be represented by a cross table. The elements on the left-hand side (row labels) are the objects, the elements at the top (column labels) are the attributes, and the incidence relation is represented by the crosses.

A central notion of FCA are the following two derivation operators
\begin{align*}
A'&:=\set{m\in M\mid \forall g\in A\colon (g,m)\in I}&&\text{for }A\subseteq G,\\
B'&:=\set{g\in G \mid \forall m\in B\colon (g,m)\in I}&&\text{for }B\subseteq M.
\end{align*}
So~$A'$ is the set of attributes common to all the objects in~$A$, and~$B'$ is the set of objects possessing the attributes in~$B$. The pair $({'},{'})$ of mappings establishes a Galois connection between the power set lattices on~$G$ and~$M$, induced by~$I$.

Let $A\subseteq G$ and $B\subseteq M$. A pair $\apply{A,B}$ is a \emph{formal concept} of $\apply{G,M,I}$ if and only if $A'=B$ and $B'=A$.~$A$ is called the \emph{extent} and~$B$ the \emph{intent} of the concept $\apply{A,B}$. The concepts of a given context can be ordered as follows:
\begin{align*}
\apply{A_1,B_1}\leq \apply{A_2,B_2}& \:\mathrel{\mathop:}\Longleftrightarrow\: A_1\subseteq A_2&\text{(or equivalently}&\iff B_2\subseteq B_1\text{)}.
\end{align*}

The ordered set of all formal concepts of $\apply{G,M,I}$ is denoted by $\underline{\mathfrak{B}}\apply{G,M,I}$ and forms a complete lattice, called the \emph{concept lattice} of $\apply{G,M,I}$. One of the fundamental theorems of formal concept analysis is the statement, shown in the monograph~\cite{ganter} by Bernhard Ganter and Rudolf Wille, that the concept lattices are, up to isomorphism, exactly the complete lattices. Every concept lattice is complete, and every complete lattice is isomorphic to some concept lattice. This applies in particular to finite lattices~$\mathbb{L}$, such as~$\Lp{n}$ for $n\in\Nplus$, which are always complete and are usually described by their \emph{standard context} $\apply{\mathcal{J}\apply{\mathbb{L}}, \mathcal{M}\apply{\mathbb{L}},\leq}$. One easily verifies that every concept lattice of a finite context and every finite lattice is \emph{doubly founded}, see~\cite{ganter} for the definition. Moreover, in~\cite{ganter} it is proved that if~$L$ is a doubly founded complete lattice, then~$\mathbb{L}$ is isomorphic to $\underline{\mathfrak{B}}\apply{\mathcal{J}\apply{\mathbb{L}},\mathcal{M}\apply{\mathbb{L}},\leq}$.

Studying the lattice~$\Lp{n}$ of positive integer partitions through its standard context immediately leads to the problem of describing the join-irreducible elements of~$\Lp{n}$. Again due to finiteness, this poses the question which $\vec{\alpha}\in\Lp{n}$ have the property that the conditions in Theorem~\ref{covering} are fulfilled by exactly one other partition~$\vec{\beta}$, in other words, where $\mathcal{D}(\vec{\alpha})$ is a singleton. As it turns out, it is not obvious how to describe and count the set of such partitions~$\vec{\alpha}$ directly. Instead we shall try to understand how to derive $\mathcal{J}\apply{\Lp{n+1}}$ from $\mathcal{J}\apply{\Lp{n}}$ for $n\in\Nplus$.

\section{Relation between the lattices~$\Lp{n}$ and $\Lp{n+1}$}\label{sect:Ln-Ln+1}

From a partition $\vec{\alpha}=\tup{a}{n}$ of $n\in \Nplus$ we obtain a new partition of $n+1$, if we add one grain to its first column, i.e.,
$\vec{\alpha}^{\downarrow_1} := \apply{a_1+1, a_2,\ldots a_n,0}$.  For $i\in \set{1,\ldots,n}$ we denote by
$\vec{\alpha}^{\downarrow_i}$ the tuple $\apply{a_1,a_2,\ldots,a_i+1,\ldots,a_n,0}$, which is not always a partition of $n+1$, and by $\vec{\alpha}^{\downarrow_{n+1}}$ the tuple $\apply{a_1,a_2,\ldots,a_n,1}$. If $A\subseteq \Lp{n}$, then
$A^{\downarrow_i}=\set{\vec{\alpha}^{\downarrow_i} \mid \vec{\alpha}\in A}$. In Figure~\ref{Lattices_L6_and_L7} the lattice diagrams of~$\Lp{6}$ and~$\Lp{7}$ are shown, the elements of $\Lp{6}^{\downarrow_1}$ appear shaded in~$\Lp{7}$.
\begin{figure}[h]
\centering
			\tikzstyle{normal}=[circle,draw,inner sep=.065cm]
			\begin{tikzpicture}[scale=0.65]
			\node[normal,label=right:{(6)},fill=gray] (16) at (-7, -1) {};
			\node[normal,label=right:{(5,1)},fill=gray] (17) at (-7, -2) {};
			\node[normal,label=right:{(4,2)},fill=gray] (18) at (-7, -3) {};
			\node[normal,label=left:{(4,1,1)},fill=gray] (19) at (-8, -4) {};
			\node[normal,label=right:{(3,3)},fill=gray] (20) at (-6, -4) {};
			\node[normal,label=right:{(3,2,1)},fill=gray] (21) at (-7, -5) {};
			\node[normal,label=left:{(3,1,1,1)},fill=gray] (22) at (-8, -6) {};
			\node[normal,label=right:{(2,2,2)},fill=gray] (23) at (-6, -6) {};
			\node[normal,label=right:{(2,2,1,1)},fill=gray] (24) at (-7, -7) {};
			\node[normal,label=right:{(2,1,1,1,1)},fill=gray] (25) at (-7, -8) {};
			\node[normal,label=right:{(1,1,1,1,1)},fill=gray] (26) at (-7, -9) {};

			\node[normal,label=right:{(7)},fill=gray] (1) at (0, 0) {};
			\node[normal,label=right:{(6,1)},fill=gray] (2) at (0, -1) {};
			\node[normal,label=right:{(5,2)},fill=gray] (3) at (0, -2) {};
			\node[normal,label=left:{(5,1,1)},fill=gray] (4) at (-1, -3) {};
			\node[normal,label=right:{(4,3)},fill=gray] (5) at (1, -3) {};
			\node[normal,label=right:{(4,2,1)},fill=gray] (6) at (0, -4) {};
			\node[normal,label=right:{(3,3,1)}] (7) at (1, -5) {};
			\node[normal,label=right:{(3,2,2)},fill=gray] (8) at (1, -6) {};
			\node[normal,label=left:{(4,1,1,1)},fill=gray] (9) at (-1, -5.5) {};
			\node[normal,label=right:{(3,2,1,1)},fill=gray] (10) at (0, -7) {};
			\node[normal,label=left:{(3,1,1,1,1)},fill=gray] (11) at (-1, -8) {};
			\node[normal,label=right:{(2,2,2,1)}] (12) at (1, -8) {};
			\node[normal,label=right:{(2,2,1,1,1)}] (13) at (0, -9) {};
			\node[normal,label=right:{(2,1,1,1,1,1)},fill=gray] (14) at (0, -10) {};
			\node[normal,label=right:{(1,1,1,1,1,1,1)}] (15) at (0, -11) {};
			\draw (1) -- (2);
			\draw (2) -- (3);
			\draw (3) -- (4);
			\draw (3) -- (5);
			\draw (4) -- (6);
			\draw (5) -- (6);
			\draw (6) -- (7);
			\draw (7) -- (8);
			\draw (6) -- (9);
			\draw (8) -- (10);
			\draw (9) -- (10);
			\draw (10) -- (11);
			\draw (10) -- (12);
			\draw (11) -- (13);
			\draw (12) -- (13);
			\draw (13) -- (14);
			\draw (14) -- (15);
			
			\draw (16) -- (17);
			\draw (17) -- (18);
			\draw (18) -- (19);
			\draw (18) -- (20);
			\draw (19) -- (21);
			\draw (20) -- (21);
			\draw (21) -- (22);
			\draw (21) -- (23);
			\draw (22) -- (24);
			\draw (23) -- (24);
			\draw (24) -- (25);
			\draw (25) -- (26);
			
			\end{tikzpicture}
\caption{Lattices $\Lp{6}$ and $\Lp{7}$}
\label{Lattices_L6_and_L7}
\end{figure}

Subsequently we prove that there is an isomorphic copy of~$\Lp{n}$ in~$\Lp{n+1}$.

\begin{lemma}\label{Lemma_isomorphic}
$\Lp{n}$ is isomorphic to $\Lp{n}^{\downarrow_1}$ for all $n\geq 1$.
\end{lemma}
\begin{proof}
Let ${\varphi\colon\Lp{n}\longrightarrow \Lp{n}^{\downarrow_1}}$ be given by ${\varphi(\vec{\alpha})=\vec{\alpha}^{\downarrow_1}}$. Clearly~$\varphi$ is injective. Since $\varphi(\Lp{n})=\Lp{n}^{\downarrow_1}$, we get that~$\varphi$ is bijective. To show that~$\varphi$ is an order-isomorphism, let $\vec{\alpha}=(a_1,\ldots,a_n)$ and $\vec{\beta}=(b_1,\ldots,b_n)$ be in~$\Lp{n}$. Then $\vec{\alpha}^{\downarrow_1}=(a'_1,\ldots,a'_{n+1})$, $\vec{\beta}^{\downarrow_1}=(b'_1,\ldots,b'_{n+1})$, where $a'_1=a_1+1$, $b'_1=b_1+1$, $a'_{n+1}=b'_{n+1}=0$ and $a'_i=a_i$, $b'_i=b_i$ for all $i\in\set{2\ldots,n}$, and it follows that
		\begin{align*}
		\vec{\alpha} \leq \vec{\beta} &\iff \sum_{i=1}^{j} a_i \leq \sum_{i=1}^{j} b_i \text{ for all } j\in \set{1,\ldots,n}\\
		&\iff 1+\sum_{i=1}^{j} a_i \leq 1+\sum_{i=1}^{j} b_i \text{ for all } j\in \set{1,\ldots,n} \\
		&\iff \sum_{i=1}^{j} a'_i \leq \sum_{i=1}^{j} b'_i \text{ for all } j\in \set{1,\ldots,n+1} \iff \vec{\alpha}^{\downarrow_1} \leq \vec{\beta}^{\downarrow_1}.\mbox{\qed}
		\end{align*}
\end{proof}

\begin{proposition}[\cite{LATAPY}\label{Prop_Latapy}]
The set $\Lp{n}^{\downarrow_1}$ forms a sublattice of $\Lp{n+1}$ for all $n\geq 1$.
\end{proposition}

We write ${\mathbb{K}} \rightarrowtail {\mathbb{L}}$ to indicate that the lattice~${\mathbb{L}}$ has a sublattice isomorphic to the lattice~${\mathbb{K}}$. This embedding relation is transitive, that is, if ${\mathbb{K}} \rightarrowtail {\mathbb{L}} \rightarrowtail {\mathbb{M}}$, then ${\mathbb{K}}\rightarrowtail {\mathbb{M}}$. Moreover, from Lemma~\ref{Lemma_isomorphic} and Proposition~\ref{Prop_Latapy} we obtain $\Lp{n} \rightarrowtail \Lp{n+1}$ for all $n\in\Nplus$.

\begin{lemma}
If $n\geq 7$, then~$\Lp{n}$ is non-modular and non-distributive.
\end{lemma}
\begin{proof}
In Figure~\ref{Lattices_L6_and_L7} the set  $\set{(4,2,1),(4,1,1,1),(3,3,1),(3,2,2),(3,2,1,1)} \subseteq \Lp{7}$ is a sublattice isomorphic to $\NNNNN$. Consequently, for $n\geq 7$ we have by Lemma~\ref{Lemma_isomorphic} and Proposition~\ref{Prop_Latapy} that $\NNNNN \rightarrowtail \Lp{7}\rightarrowtail \Lp{8}\rightarrowtail \ldots\rightarrowtail \Lp{n-1}\rightarrowtail \Lp{n}$. By transitivity it follows that $\NNNNN\rightarrowtail\Lp{n}$. Hence, applying the $\MMM$-$\NNNNN$-Theorem in~\cite[p.~89]{davey} to~$\Lp{n}$, we obtain that it is non-modular and non-distributive.\qed
\end{proof}

If $1\leq n\leq 5$, the lattices~$\Lp{n}$ form a chain under the dominance ordering.  
From this and Figure~\ref{Lattices_L6_and_L7} we see that for $n\in\set{1,\dotsc,6}$, the lattice~$\Lp{n}$ does
not have sublattices isomorphic to~$\NNNNN$ or~$\MMM$, hence it is distributive and thus modular.

In what follows, the process to obtain the lattice~$\Lp{n+1}$ from~$\Lp{n}$ described in~\cite{LATAPY} is shown. For their construction the authors of~\cite{LATAPY} analyzed the directly reachable partitions from a partition $\vec{\alpha} \in \Lp{n}$, considering the following sets.\par

Let $n\in \Nplus$. The \emph{set of all partitions} of~$n$ with a \emph{cliff at~$1$} is denoted by $\Cl(n)$, with a \emph{slippery step at~$1$} by $\Ss(n)$, with a \emph{non-slippery step at~$1$} by $\NS(n)$, with a \emph{slippery plateau of length~$l\geq 1$ at~$1$} by $\PP_l(n)$, and with a \emph{non-slippery plateau at~$1$} by $\NP(n)$.

In order to generate the elements of $\Lp{n+1}$ from~$\Lp{n}$, we follow the two steps given in~\cite{LATAPY}. For each $\vec{\alpha} \in \Lp{n}$ we perform:
\begin{description}
\item[Step 1] Add one grain to the first column, that is, construct $\vec{\alpha}^{\downarrow 1}$.
\item[Step 2] If $\vec{\alpha} \in \Ss(n)$ or $\vec{\alpha}\in \NS(n)$, then construct $\vec{\alpha}^{\downarrow 2}$. If $\vec{\alpha}\in \PP_l(n)$ for some $l\geq 1$, then construct $\vec{\alpha}^{\downarrow l+2}$.
\end{description}

From these two steps every partition $\vec{\alpha}\in \Lp{n}$ generates at most two elements of $\Lp{n+1}$, called the \emph{sons} of~$\vec{\alpha}$. The element described in the first step, $\vec{\alpha}^{\downarrow 1}$ is called the \emph{left son} and the one in the second step is the \emph{right son} (if it exists). If~$\vec{\beta}$ is a son of~$\vec{\alpha}$, we call~$\vec{\alpha}$ \emph{the father} of~$\vec{\beta}$.

By the following theorem from~\cite{LATAPY}, every element in~$\Lp{n+1}$ has a father in~$\Lp{n}$.

\begin{theorem}[Latapy and Phan~\cite{LATAPY}]\label{recur}
For all $n\geq 1$, we have:
		\[
		\Lp{n+1}=\Lp{n}^{\downarrow_1} \sqcup \Ss(n)^{\downarrow_2} \sqcup \NS(n)^{\downarrow_2}\sqcup \bigsqcup_{1\leq l<n} \PP_l(n)^{\downarrow_{l+2}}.
		\]
\end{theorem}

Because these unions are disjoint and because the map $\vec{\alpha}\mapsto \vec{\alpha}^{\downarrow i}$ is injective, we have that every $\vec{\beta}\in \Lp{n+1}$ possesses a unique father in~$\Lp{n}$. Using this close connection between~$\Lp{n}$ and~$\Lp{n+1}$, in the next section we shall generate the elements of $\mathcal{J}\apply{\Lp{n+1}}$ from elements of $\mathcal{J}\apply{\Lp{n}}$ via the two steps described above. This will allow us to determine the number of objects and attributes of the standard context of~$\Lp{n}$.
\section{Standard context of~$\Lp{n}$}\label{sect:std-cxt}

In this section we study the standard context of~$\Lp{n}$, which is a restricted version of the context  $\apply{\Part(n),\Part(n),\leq}$. This restriction provides us with a smaller number of objects and attributes, and, most importantly, from these proper subsets of~$\Lp{n}$ one can recover the full information about the structure of~$\Lp{n}$. This is so since for a finite lattice~$\mathbb{L}$ one can prove that~$\mathbb{L}$ is isomorphic to the concept lattice $\underline{\mathfrak{B}}\apply{{\mathcal{J}}\apply{\mathbb{L}}, \mathcal{M}\apply{\mathbb{L}}, {\leq} \cap \apply{\mathcal{J}\apply{\mathbb{L}}\times \mathcal{M}\apply{\mathbb{L}}}}$, and hence our motivation to investigate the standard context of~$\Lp{n}$ in more detail.

\begin{example}\label{ex:standard-context-L6}
		We calculate the standard context of $\Lp{6}$, the diagram of which is shown in Figure~\ref{Lattices_L6_and_L7}. We have that
		\begin{align*}
		\mathcal{J}(\Lp{6})= \{&(2,1,1,1,1),(2,2,1,1),(2,2,2),(3,1,1,1),(3,3),(4,1,1),(5,1),(6)\} \\
		\mathcal{M}(\Lp{6})= \{&(1,1,1,1,1,1),(2,1,1,1,1),(2,2,2),(3,1,1,1),(3,3),(4,1,1),\\
		&(4,2),(5,1)\}.
		\end{align*}
Thus, the standard context, denoted by $\mathbb{K}\apply{\Lp{6}}$, is:
\begin{center}
\small
\begin{tabular}{|c||c|c|c|c|c|c|c|c|}\hline
$\mathbb{K}$&111111&$21111$&$222$&$3111$&$33$&$411$&$42$&$51$\\ \hline\hline
$21111$& &$\times$&$\times$&$\times$&$\times$&$\times$&$\times$&$\times$\\ \hline
$2211$& &&$\times$&$\times$&$\times$&$\times$&$\times$&$\times$\\ \hline
$222$& &&$\times$&&$\times$&$\times$&$\times$&$\times$\\ \hline
$3111$& &&&$\times$&$\times$&$\times$&$\times$&$\times$\\ \hline
$33$& &&&&$\times$&&$\times$&$\times$\\ \hline
$411$& &&&&&$\times$&$\times$&$\times$\\ \hline
$51$& &&&&&&&$\times$\\ \hline
$6$& &&&&&&&\\ \hline
\end{tabular}
\end{center}

Observe that while the context $\apply{\Lp{6},\Lp{6},\leq}$ has $11$ objects and $11$ attributes the standard context has only~$8$ of each. This advantage in size becomes much more pronounced for larger values of~$n$ than just~$n=6$, as we shall see below.
\end{example}

Determining the number of elements in~$\Lp{n}$ is a problem that has captivated mathematicians over the years. Whether our approach via formal concept analysis and the standard context $\mathbb{K}\apply{\Lp{n}}$ is to offer any advantage over working with~$\Lp{n}$ directly, remains, however, unclear until we show that~$\mathbb{K}\apply{\Lp{n}}$ stays much more manageable in size than~$\Lp{n}$ when~$n$ increases.
Therefore, it is natural to ask about the number of elements which are $\vee$-irreducible or $\wedge$-irreducible. To study the growth of the standard context, it is, in fact, sufficient to know how the number $|{\mathcal{J}}\apply{\Lp{n}}|$ is growing, since partition conjugation $^*$ is a lattice antiautomorphism~\cite{Brylawski}, and hence we have ${\mathcal{J}}\apply{\Lp{n}}^*= {\mathcal{M}}\apply{\Lp{n}}$ and thus $|{\mathcal{J}}\apply{\Lp{n}}|=|{\mathcal{M}}\apply{\Lp{n}}|$. In the following table we present $|{\mathcal{J}}\apply{\Lp{n}}|$ for $1\leq n\leq 8$. The apparent pattern in the last column suggests a relation between  $|{\mathcal{J}}\apply{\Lp{n}}|$  and $|{\mathcal{J}}\apply{\Lp{n+1}}|$.
\bgroup\small
\[\begin{array}{|*{5}{c|}}
				\hline
				\multicolumn{1}{|l|}{n} & \multicolumn{1}{l|}{|\mathcal{L}_n|} & \multicolumn{1}{l|}{|\mathcal{M}(\Lp{n})|} & \multicolumn{1}{l|}{|\mathcal{J}(\Lp{n})|} & \multicolumn{1}{l|}{|\mathcal{J}(\Lp{n+1})|-|\mathcal{J}(\Lp{n})|} \\ \hline
				1 & 1 & 0 & 0 & 1\\ \hline
				2 & 2 & 1 & 1 & 1 \\ \hline
				3 & 3 & 2 & 2 & 2 \\ \hline
				4 & 5 & 4 & 4 & 2\\ \hline
				5 & 7 & 6 & 6 & 2\\ \hline
				6 & 11 & 8 & 8 & 3 \\ \hline
				7 & 15 & 11 & 11 & 3 \\ \hline
				8 & 22 & 14 & 14 & 3\\ \hline
\end{array}\]
\egroup

For a finite lattice the join-irreducible elements can be characterized as those that cover precisely one element. In~$\Lp{n}$ these are those that have exactly one cliff (and no slippery step) or exactly one slippery step (and no cliff). In~\cite{Brylawski} Brylawski characterized the $\vee$-irreducible elements as follows.

\begin{lemma}[{\cite[Corollary 2.5]{Brylawski}}]\label{lem:char-join-irr}
For $n \geq 1$ the join-irreducible partitions from ${\mathcal{J}}\apply{\Lp{n}}$ can be categorized into four types where always $m,l,s\geq 1$:
\begin{description}
\item[Type~A:] $\tupone{k}{m}$ for $k\geq 2$.
\item[Type~B:] $\tupdos{k}{m}{k-1}{m+l}$ for $k\geq 2$.
\item[Type~C:] $\tupdos{k}{m}{1}{m+l}$ for $k\geq 3$.
\item[Type~D:] $\tuptres{k+2}{m}{k+1}{m+l}{1}{m+l+s}$ for $k\geq 2$.
\end{description}
\end{lemma}

When we apply to the elements of ${\mathcal{J}}\apply{\Lp{n}}$ the two steps described in Section~\ref{sect:Ln-Ln+1} to generate elements of $\Lp{n+1}$ from~$\Lp{n}$, we obtain the following proposition.

\begin{proposition}\label{Prop_Irreducible_Lp}
Let $n\geq 3$. For every partition $\vec{\alpha}\in {\mathcal{J}(\Lp{n})\setminus \set{(2,1,\ldots,1)}}$ among its (at most two) sons there is exactly one that belongs to $\mathcal{J}(\Lp{n+1})$; moreover, the partition $(2,1,\ldots,1) \in \mathcal{J}(\Lp{n})$ has two sons in $\mathcal{J}(\Lp{n+1})$.
\end{proposition}
\begin{proof}
Every $\vec{\alpha}\in {\mathcal{J}}\apply{\Lp{n}}$ belongs to one of the four types from Lemma~\ref{lem:char-join-irr}. So we apply to $\vec{\alpha}$ the two steps from Section~\ref{sect:Ln-Ln+1} and analyze the resulting partitions.
 \begin{description}
  \item[Case 1:] If $\vec{\alpha}$ is a partition of type~A, then we consider two subcases:
			\begin{itemize}
				\item If $m=1$, then $\vec{\alpha}=(k)$ has a cliff at~$1$. Hence, it has the left son $\vec{\alpha}^{\downarrow_1}=(k+1)=(n+1)$,
				 which is a partition of type~A, but it has no right son. So we have $\vec{\alpha}^{\downarrow_1}\in {\mathcal{J}}\apply{\Lp{n+1}}$.
				\item If $m\geq 2$, then $\vec{\alpha}=\tupone{k}{m}$ has a non-slippery plateau at~$1$, hence it has only the left son $\vec{\alpha}^{\downarrow_1}=(k+1,k,\ldots,\stackrel{m}{k})$, which is a partition of type~B. Thus, $\vec{\alpha}^{\downarrow_1}\in \mathcal{J}(\Lp{n+1})$.  	
			\end{itemize}
 \item[Case 2:] If $\vec{\alpha}$ is a partition of type~B, then we consider three subcases:
      \begin{itemize}
 			  \item If $m=1, l\geq 1, k\geq 3$, then $\vec{\alpha}=(k,k-1,\ldots,\stackrel{1+l}{k-1})$ has a non-slippery step at~$1$. Hence, this partition has two distinct sons. The left one is $\vec{\alpha}^{\downarrow_1}={(k+1,k-1,\ldots,\stackrel{1+l}{k-1})}$, which is not $\vee$-irreducible, and the right one is $\vec{\alpha}^{\downarrow_{2}}=(k,k,k-1,\ldots,\stackrel{1+l}{k-1})$, which belongs to ${\mathcal{J}}\apply{\Lp{n+1}}$ because it is of type~B (if $l\geq 2$) or of type~A (if $l=1$).
			  \item If $m=1, l\geq 1, k= 2$, then $\vec{\alpha}=(2,1,\ldots,\stackrel{1+l}{1})$ has a slippery step at~$1$. Hence, the left son $\vec{\alpha}^{\downarrow_1}=(3,1,\ldots,\stackrel{1+l}{1})$ belongs to ${\mathcal{J}}\apply{\Lp{n+1}}$ because it is of type~C, and the right son $\vec{\alpha}^{\downarrow_{2}}=(2,2,1,\ldots,\stackrel{1+l}{1})$ also belongs to ${\mathcal{J}}\apply{\Lp{n+1}}$ because it is of type~B (if $l\geq 2$) or of type~A (if $l=1$).
			  \item If $m\geq 2, l\geq 1, k\geq 2$, then $\vec{\alpha}=\tupdos{k}{m}{k-1}{m+l} $ has a slippery plateau at~$1$ of length $m-1$, thus it has two sons. The left son $\vec{\alpha}^{\downarrow_1}=(k+1,k,\ldots,\stackrel{m}{k},k-1,\ldots,\stackrel{m+1}{k-1})$ is not  $\vee$-irreducible because for $k\geq 3$ it does not belong to any type described by Brylawski, and if $k=2$, then  $\vec{\alpha}^{\downarrow_1}=(3,2,\ldots,\stackrel{m}{2},1,\ldots,\stackrel{m+1}{1})$ is not  $\vee$-irreducible either, despite looking a bit like the form of type~$D$ (however it is not because for partitions of type~$D$ the first entry of the partition is greater or equal to~$4$). The right son
 $\vec{\alpha}^{\downarrow_{m+1}}=(k,\ldots,\stackrel{m}{k},\stackrel{m+1}{k},k-1,\ldots,\stackrel{m+l}{k-1})$ belongs to ${\mathcal{J}}\apply{\Lp{n+1}}$ because it is of type~B (if $l\geq 1$) or of type~A (if $l=1$).
 			\end{itemize}

  \item[Case 3:] If $\vec{\alpha}$ is a partition of type~C, then we consider two subcases:
    \begin{itemize}
 			\item If $m=1, l\geq 1, k\geq 3$, then $\vec{\alpha}$ has a cliff at~$1$, and it has only the left son $\vec{\alpha}^{\downarrow_1}=(k+1,1,\ldots,\stackrel{1+l}{1})$, belonging to ${\mathcal{J}}\apply{\Lp{n+1}}$ because it is of type C.
			\item If $m\geq 2, l\geq 1, k\geq 3$, then $\vec{\alpha}$ has a non-slippery plateau at~$1$. Hence, it only has the left son $\vec{\alpha}^{\downarrow_1}=(k+1,k,\ldots,\stackrel{m}{k},1,\ldots,\stackrel{m+l}{1})$, which is of type~D. Thus, $\vec{\alpha}^{\downarrow_1}\in {\mathcal{J}}\apply{\Lp{n+1}}$.
 			\end{itemize}

  \item[Case 4:] If $\vec{\alpha}$ is a partition of type~D, then we consider two subcases:
    \begin{itemize}
 			\item If $m\geq 2, l,s\geq 1, k\geq 2$, then $\vec{\alpha}$ has a slippery plateau at~$1$ of length $m-1$, hence it has two sons.\par
 			The left son $\vec{\alpha}^{\downarrow_1}=(k+3,k+2,\ldots,\stackrel{m}{k+2},k+1,\ldots,\stackrel{m+l}{k+1},1,\ldots,\stackrel{m+l+s}{1})$ is not  $\vee$-irreducible. The right son
\[\vec{\alpha}^{\downarrow_{m+1}}=(k+2,\ldots,\stackrel{m}{k+2},\stackrel{m+1}{k+2},k+1\ldots,\stackrel{m+l}{k+1},1,\ldots,\stackrel{m+l+s}{1})\]
belongs to ${\mathcal{J}}\apply{\Lp{n+1}}$, because it is of type~D (if $l\geq 2, s\geq 1$) or of type~C (if $l=1,s\geq 1$). 			
			\item If $m=1, l\geq 1, k\geq 2$, then $\vec{\alpha}$ has a non-slippery step at~$1$. Hence, the left son $\vec{\alpha}^{\downarrow_1}=(k+3,k+1,\ldots,\stackrel{m+l}{k+1},1,\ldots,\stackrel{m+l+s}{1})$ is not $\vee$-irreducible, and the right son
$\vec{\alpha}^{\downarrow_{2}}=(k+2,k+2,k+1,\ldots,\stackrel{m+l}{k+1},1,\ldots,\stackrel{m+l+s}{1})$ belongs to $ {\mathcal{J}}\apply{\Lp{n+1}}$
because it is of type~D (if $l\geq 2, s\geq 1$) or of type~C (if $l=1,s\geq 1$). 
			\end{itemize}
\end{description}

This shows that every element $\vec{\alpha}\in  {\mathcal{J}}\apply{\Lp{n}}\setminus\set{(2,1,\dotsc,1)}$ has precisely one son which belongs to $\mathcal{J}\apply{\Lp{n+1}}$ and the only partition that has two sons belonging to $ {\mathcal{J}}\apply{\Lp{n+1}}$ is $\apply{2,1,\ldots,1}$.\qed
\end{proof}

Exploiting Proposition~\ref{Prop_Irreducible_Lp}, we can now define a map $\eta\colon\mathcal{J}(\Lp{n})\to \mathcal{J}(\Lp{n+1})$ as follows:
if $\vec{\alpha} \in  \mathcal{J}(\Lp{n})\setminus \set{(2,1,\ldots,1)}$, denote by $\eta(\vec{\alpha})$ the unique son of~$\vec{\alpha}$ that belongs to~$\mathcal{J}\apply{\Lp{n+1}}$; moreover, let $\eta(2,1,\ldots,\stackrel{l+1}{1}):= (2, 2, 1,\ldots,\stackrel{l+1}{1})$.
Although the partition $(2,1,\ldots,\stackrel{l+1}{1})$ has two sons that belong to ${\mathcal{J}}\apply{\Lp{n+1}}$, it is convenient to ignore the left son in the definition of~$\eta$.
From the proof of Proposition~\ref{Prop_Irreducible_Lp}, we can work out an explicit expression for~$\eta$.
For any $\vec{\alpha}\in {\mathcal{J}}(\Lp{n})$ we have
\begin{align}\label{defex}
	\eta(\vec{\alpha})=\begin{cases}
	\vec{\alpha}^{\downarrow_1} &\text{ if } \vec{\alpha} \in \Cl(n) \cup \NP(n),\\
	\vec{\alpha}^{\downarrow_2} &\text{ if } \vec{\alpha} \in \Ss(n) \cup \NS(n),\\
	\vec{\alpha}^{\downarrow_{l+2}} &\text{ if } \vec{\alpha} \in \PP_l(n) \text{ for some } 1\leq l<n.
	\end{cases}
	\end{align}
Since~$\eta$ is injective we have $\left|\eta\apply{{\mathcal{J}}\apply{\Lp{n}}}\right|=\left|{\mathcal{J}}\apply{\Lp{n}}\right|$. Moreover, we have
\[\left| {\mathcal{J}}\apply{\Lp{n+1}}\right|=\left|\eta\apply{{\mathcal{J}}\apply{\Lp{n}}}\right|+ \left|{\mathcal{J}}\apply{\Lp{n+1}} \setminus \eta\apply{{\mathcal{J}}\apply{\Lp{n}}}\right|.\]

To calculate $\left| {\mathcal{J}}\apply{\Lp{n+1}}\right|$, we have to identify the elements of ${\mathcal{J}}\apply{\Lp{n+1}}$ that are not in the image of~$\eta$.

\begin{lemma}\label{imeta}
For $n\geq3$ we have $\mathcal{J}(\Lp{n+1}) \setminus \eta(\mathcal{J}(\Lp{n}))=E_1 \cup E_2$, where~$E_1$ and~$E_2$ are the following exceptional sets of partitions of~$n+1$:
\begin{align*}
  E_1&:=\{(2,1,\ldots,\stackrel{n}{1})\},&
  E_2&:=\{(3,\ldots,\stackrel{m}{3},1,\ldots,\stackrel{m+l}{1}) \mid m\geq 1, l\geq 1 \}.
\end{align*}
\end{lemma}

\begin{proof}
Let $\vec{\beta}\in {\mathcal{J}}\apply{\Lp{n+1}}$. Then, similarly to Proposition~\ref{Prop_Irreducible_Lp}, we consider several cases, according to the four types of $\vee$-irreducible partitions described by Brylawski, cf.\ Lemma~\ref{lem:char-join-irr}.
 \begin{description}
  \item[Case 1:] If~$\vec{\beta}$ is of the form $\tupone{k}{m}$, with $k\geq 2$ and $m\geq 1$, then we consider two sub-cases:
\begin{itemize}
\item If $m\geq 2$, then~$\vec{\beta}$ is a son of $(k\ldots,k,\stackrel{m}{k-1}) \in {\mathcal{J}}\apply{\Lp{n}}$.
\item If $m=1$, then $\vec{\beta}=(n+1)$ is a son of $(n)\in {\mathcal{J}}\apply{\mathcal{L}_n}$.
\end{itemize}
\item[Case 2:] If $\vec{\beta}$ is of the form $(k,\ldots,\stackrel{m}{k},k-1,\ldots, \stackrel{m+l}{k-1})$, with $k\geq 2$ and $m,l\geq 1$, then we consider three sub-cases:
\begin{itemize}
\item If $m=1$, $l\geq 1$, and $k\geq 3$, then $\vec{\beta}$ is a son of $(k-1,\ldots,k-1) \in {\mathcal{J}}\apply{\Lp{n}}$.
\item If $m=1$, $l\geq 1$, and $k= 2$, then $\vec{\beta}=(2,1,\ldots,1)$ is a son of $(1,1,\ldots,1)$, which does not belong to $\mathcal{J}(\Lp{n})$, but~$\vec{\beta}$ does not arise from any of the four cases of Proposition~\ref{Prop_Irreducible_Lp}.
\item If $m\geq 2$, $l\geq 1$, and $k\geq 2$, then $\vec{\beta}$ is a son of $(k,\ldots\!\stackrel{m-1}{,k,}\!k-1,\ldots,k-1)$ from $\mathcal{J}(\Lp{n})$.
\end{itemize}
\item[Case 3:] If $\vec{\beta}$ is of the form $(k,\ldots,\stackrel{m}{k},1,\ldots,\stackrel{m+l}{1})$, with $k\geq 2$ and $m,l\geq 1$, then we consider three sub-cases:
\begin{itemize}
\item If $m=1$, $l\geq 1$, and $k\geq 3$, then $\vec{\beta}$ is the left son of $(k-1,1,\ldots,\stackrel{1+l}{1})$. For $k\geq 4$ it belongs to ${\mathcal{J}}\apply{\Lp{n}}$, so $\vec{\beta}\in\eta\apply{\mathcal{J}\apply{\Lp{n}}}$. However, $\vec{\beta}\notin\eta\apply{\mathcal{J}\apply{\Lp{n}}}$ when $k=3$ as~$\eta$ chooses the right son of $(2,1\ldots,1)$.
\item If $m\geq 2$, $l\geq 1$, and $k\geq 4$, then $\vec{\beta}$ is the right son of the type~D partition $(k,\ldots,\stackrel{m-1}{k},\stackrel{m}{k-1},1,\ldots,\stackrel{m+l}{1}) \in \mathcal{J}(\Lp{n})$.
\item If $m\geq 2$, $l\geq 1$, and $k=3$, then $\vec{\beta}=(3,\ldots,\stackrel{m}{3},1,\ldots,\stackrel{m+l}{1})$ is the right son of $(3,\ldots,\stackrel{m-1}{3},\stackrel{m}{2},1,\ldots,\stackrel{m+l}{1}) \notin \mathcal{J}(\Lp{n})$,
but~$\beta$ is not obtained from any of the four cases (in the proof) of Proposition~\ref{Prop_Irreducible_Lp}.
\end{itemize}
\item[Case 4:] If $\vec{\beta}$ is of the form $(k+2,\ldots,\stackrel{m}{k+2},k+1,\ldots,\stackrel{m+l}{k+1},1,\ldots,\stackrel{m+l+s}{1})$ with $k\geq 2$ and $m,l,s\geq 1$, then we consider two sub-cases:
\begin{itemize}
\item If $m=1$ and $l,s\geq 1$, then~$\vec{\beta}$ is the left son of $(k+1,\ldots,\stackrel{1+l}{k+1},1,\ldots,\!\!\!\!\stackrel{1+l+s}{1})$, which belongs to ${\mathcal{J}}\apply{\mathcal{L}_n}$.
\item If $m\geq 2$ and $l,s\geq 1$, then $\vec{\beta}$ is the right son of the type~D partition $(k+2,\ldots,\stackrel{m-1}{k+2},\stackrel{m}{k+1},\ldots,\stackrel{m+l}{k+1},1,\ldots,\stackrel{m+l+s}{1})\in{\mathcal{J}}\apply{\mathcal{L}_n}$.
\end{itemize}
\end{description}

Thus, almost all elements of $J\apply{\mathcal{L}_{n+1}}$ are sons of some element in $J\apply{\mathcal{L}_n}$ except for those of the form $(3,3,\ldots,\stackrel{m}{3},1,\ldots,\stackrel{m+l}{1})$, with $m\geq 2$ and $l\geq 1$, or the partition $\apply{2,1,\ldots,1}$. Additionally, the partition $\apply{3,1,\ldots,1}$, which we could get from a father in ${\mathcal{J}}\apply{\Lp{n}}$, was excluded from the image of~$\eta$ by definition. Therefore, we conclude that the elements of ${\mathcal{J}}\apply{\Lp{n+1}}$ that are not in the image of~$\eta$ are all of the form $(3,\ldots,\stackrel{m}{3},1,\ldots,\stackrel{m+l}{1})$ for $m,l\geq 1$ or they are the partition $\apply{2,1,\ldots,1}$.\qed
\end{proof}

We can now describe how to construct the set ${\mathcal{J}}\apply{\Lp{n+1}}$ from ${\mathcal{J}}\apply{\Lp{n}}$.
\begin{theorem}\label{constsup}
		Let $n\geq 3$. Then 
		\begin{multline*}
		\mathcal{J}(\Lp{n+1})= \left(\mathcal{J}(\Lp{n})\cap \Cl(n) \right)^{\downarrow_1} \sqcup \left(\mathcal{J}(\Lp{n})\cap \NP(n) \right)^{\downarrow_1}
		\sqcup \left(\mathcal{J}(\Lp{n})\cap \Ss(n) \right)^{\downarrow_2}\\  \sqcup \left(\mathcal{J}(\Lp{n})\cap \NS(n) \right)^{\downarrow_2} 
		\sqcup\bigsqcup_{1\leq l <n} \left(\mathcal{J}(\Lp{n})\cap \PP_l(n) \right)^{\downarrow_{l+2}} \sqcup E_1(n+1)  \sqcup E_2(n+1),
		\end{multline*}
%
		where $E_1(k),E_2(k)\subseteq\Lp{k}$ denote:
		\begin{align*}
		E_1(k)& =\{(2,1,\ldots,1)\},&
		E_2(k)&=\{(3,\ldots,\stackrel{m}{3},1,\ldots,\stackrel{m+l}{1}) \mid m\geq 1, l\geq 1 \}.
		\end{align*}
\end{theorem}

\begin{proof}
Obviously, $\mathcal{J}(\Lp{n+1})=\eta(\mathcal{J}(\Lp{n})) \sqcup (\mathcal{J}(\Lp{n+1})\setminus \eta(\mathcal{J}(\Lp{n})))$, and,
from the definition of~$\eta$ given in~\eqref{defex}, we have
		\begin{multline*}
		\eta(\mathcal{J}(\Lp{n}))= \left(\mathcal{J}(\Lp{n})\cap \Cl(n) \right)^{\downarrow_1} \sqcup \left(\mathcal{J}(\Lp{n})\cap \NP(n) \right)^{\downarrow_1}
		\sqcup \left(\mathcal{J}(\Lp{n})\cap \Ss(n) \right)^{\downarrow_2}\\ \sqcup \left(\mathcal{J}(\Lp{n})\cap \NS(n) \right)^{\downarrow_2}
		\sqcup\bigsqcup_{1\leq l <n} \left(\mathcal{J}(\Lp{n})\cap \PP_l(n) \right)^{\downarrow_{l+2}}.
		\end{multline*}
		By Lemma~\ref{imeta}, it follows that $\mathcal{J}(\Lp{n+1})\setminus \eta(\mathcal{J}(\Lp{n}))=E_1(n+1) \sqcup E_2(n+1)$.\qed
\end{proof}

Counting the elements of $\mathcal{J}(\Lp{n+1})$ was the main motivation to study the relationship between $\mathcal{J}(\Lp{n})$ and $\mathcal{J}(\Lp{n+1})$. To finally achieve this, we need one more result.

\begin{lemma}\label{numpart}
		Let $n\geq 3$. The number of partitions of $n+1$ of the form
    \[(3,\ldots,\stackrel{x}{3},1,\ldots,\stackrel{x+y}{1})\]
    with $x\geq 1$ and $y\geq 1$, is $\left\lfloor n/3 \right\rfloor$, i.e.,
		$|E_2(n+1)|=\left\lfloor n/3 \right\rfloor$.
	\end{lemma}
\begin{proof}
		We have a one-to-one correspondence between each partition of the form  $(3,\ldots,\stackrel{x}{3}, 1,\ldots,\stackrel{x+y}{1})$ and each integer solution of
		\begin{equation}
		3x+y=n+1 \text{ subject to } x\geq 1, y \geq 1.
		\end{equation}
		 If we subtract one from both sides in the last equation and change the variable, then we obtain
		\begin{equation}\label{sols}
		3x+y=n \text{ subject to } x\geq 1, y \geq 0.
		\end{equation}
		Thus, the number of solutions of~\eqref{sols} will be the same as the number of partitions in~$E_2(n+1)$. Moreover, we have		
		\begin{align*}
		3\left\lfloor \dfrac{n}{3} \right\rfloor+r=n, 
		\end{align*}
		where $0\leq r <3$. If $x \in \set{1,\ldots,\lfloor n/3 \rfloor}$, then ${n-3x \geq n-3 \lfloor n/3 \rfloor=r \geq 0}$. Letting $y=n-3x$, we have that $(x,y)$ is a solution of~\eqref{sols}. But if $x>\lfloor n/3 \rfloor$, then $x \geq \lfloor n/3 \rfloor+1$, implying $y=n-3x\leq {n-3(\lfloor n/3 \rfloor+1)}=r-3<0$, whence $(x,y)$ is no solution. All solutions of~\eqref{sols} are $\set{(x,y) \mid 1\leq x\leq\lfloor n/3 \rfloor, y=n-3x}$; thus, there are $\lfloor n/3 \rfloor$ of them, and therefore $\lfloor n/3 \rfloor$ partitions in $E_2(n+1)$.\qed
		\end{proof}
		
The following theorem reveals the pattern that appeared in the last column of the table shown after Example~\ref{ex:standard-context-L6}.
	\begin{theorem}\label{thm:recursion}
    Starting from $|\mathcal{J}(\Lp{1})|=0$, for every $n\geq 1$ we have the recursion
    \[
    |\mathcal{J}(\Lp{n+1})|=|\mathcal{J}(\Lp{n})|+\left\lfloor \dfrac{n}{3} \right\rfloor+1.
    \]
	\end{theorem}		
 \begin{proof} For $n=1$ and $n=2$, we have
 		\begin{align*}
		|\mathcal{J}(\Lp{2})|=1=0+0+1=|\mathcal{J}(\Lp{1})|+\lfloor 1/3 \rfloor +1, \\
		|\mathcal{J}(\Lp{3})|=2=1+0+1=|\mathcal{J}(\Lp{2})|+\lfloor 2/3 \rfloor +1.
		\end{align*}
since $|\mathcal{J}(\Lp{1})|=0$. For $n\geq 3$, since~$\eta$ is injective, we have $|\eta(\mathcal{J}(\Lp{n}))|=|\mathcal{J}(\Lp{n})|$, and Lemma~\ref{imeta} states that
		\begin{align*}
		|\mathcal{J}(\Lp{n+1})\setminus \eta(\mathcal{J}(\Lp{n}))|=|E_1(n+1) \sqcup E_2(n+1)|=|E_1(n+1)|+|E_2(n+1)|.
		\end{align*}
		Moreover, $|E_1(n+1)|=1$, and Lemma~\ref{numpart} yields ${|E_2(n+1)|=\lfloor n/3 \rfloor}$. Thus,
    \[
		|\mathcal{J}(\Lp{n+1})|=|\eta(\mathcal{J}(\Lp{n}))| + |\mathcal{J}(\Lp{n+1})\setminus \eta(\mathcal{J}(\Lp{n}))| 
		=|\mathcal{J}(\Lp{n})|+\left\lfloor \frac{n}{3} \right\rfloor+1.\mbox{\qed}
		\]
	\end{proof}

From the last theorem, we can get a closed formula for $|\mathcal{J}(\Lp{n+1})|$, which gives us a clearer picture of the cardinality of $|\mathcal{J}(\Lp{n+1})|$ in terms of~$n$.
\begin{corollary}\label{cor:number-join-irr}
		For all $n\geq 1$ we have
		\begin{align}\label{formula}
		|\mathcal{J}(\Lp{n+1})|=n \left(\left\lfloor \frac{n}{3} \right\rfloor+1\right)-\frac{3}{2}\left\lfloor \frac{n}{3} \right\rfloor^{2}-\frac{1}{2}\left\lfloor \frac{n}{3} \right\rfloor.
		\end{align}
	\end{corollary}	
	\begin{proof}
		By Theorem~\ref{thm:recursion} we have
		$|\mathcal{J}(\Lp{i+1})|-|\mathcal{J}(\Lp{i})|=\left\lfloor \dfrac{i}{3} \right\rfloor+1$
		for  $i\geq 1$. Thus,
    \begin{equation*}
|\mathcal{J}(\Lp{n+1})|=|\mathcal{J}(\Lp{n+1})|-|\mathcal{J}(\Lp{1})|=\sum_{i=1}^{n} (|\mathcal{J}(\Lp{i+1})|-|\mathcal{J}(\Lp{i})|)
 =n+\sum_{i=1}^{n} \left\lfloor \dfrac{i}{3} \right\rfloor
    \end{equation*}
		because $|\mathcal{J}(\Lp{1})|=0$. Let us calculate $\sum_{i=1}^{n} \lfloor i/3 \rfloor$. Put $q=\left\lfloor n/3 \right\rfloor$ and $r=n-3q$. Hence, $n=3q+r$ with $0\leq r<3$, and we have $\sum_{i=1}^{n} \lfloor i/3 \rfloor=\sum_{i=0}^{n} \lfloor i/3 \rfloor=u+v$, where
		\begin{align*}
		u&=\sum_{i=0}^{3q} \left\lfloor \dfrac{i}{3} \right\rfloor
     =0+0+0+1+1+1+\ldots+(q-1)+(q-1)+(q-1)+q \\
     &=3(1+\ldots+(q-1))+q=3\dfrac{q(q-1)}{2}+q=\dfrac{3}{2}q^{2}-\dfrac{1}{2}q,\\
		v&=\sum_{i=1}^{r} \left\lfloor \dfrac{3q+i}{3} \right\rfloor 
     =\underbrace{q+\ldots+ q}_{r \text{ times}}=rq=(n-3q)q=nq-3q^2.
		\end{align*}
		Thus, $
		|\mathcal{J}(\Lp{n+1})|=n+\frac{3}{2}q^{2}-\frac{1}{2}q+nq-3q^2=n(q+1)-\frac{3}{2}q^2-\frac{1}{2}q.\mbox{\qed}$
    
	\end{proof}
	
An alternative proof of the main technical result Corollary \ref{cor:number-join-irr}, written in parallel, was given independently by Bernhard Ganter~\cite{Ganter20} (not using methods from~\cite{LATAPY}).\par 
To obtain the standard context $\mathbb{K}(\Lp{n+1})$ from $\mathbb{K}(\Lp{n})$ we do the following. First, we construct the objects, i.e., we construct $\mathcal{J}(\Lp{n+1})$ from $\mathcal{J}(\Lp{n})$ as Theorem~\ref{constsup} shows. Second, we calculate $\mathcal{J}(\Lp{n+1})^{*}$ in order to obtain the attributes of $\mathbb{K}(\Lp{n+1})$. Finally, we fill the cross table using the dominance ordering as the incidence relation between objects and attributes. This can be done efficiently:

\begin{corollary}\label{cor:std-context-algo}
The set $\mathcal{J}(\Lp{n})$ can be obtained in time $\Theta(n^3)$, and
the standard context $\mathbb{K}(\Lp{n})$ can be produced in time $\Theta(n^4)$.
\end{corollary}
\begin{proof}
With constant effort we create $\mathcal{J}(\Lp{k_0})$ for an initial value $k_0\in\Nplus$ and sort its partitions according to the classes occurring in Theorem~\ref{constsup}. We then iterate over the $\Theta(k^2)$ partitions in~$\mathcal{J}(\Lp{k})$ (see Corollary~\ref{cor:number-join-irr}) plus the~$\Theta(k)$ exceptional partitions (see Lemma~\ref{numpart}) to obtain $\mathcal{J}(\Lp{k+1})$ from $\mathcal{J}(\Lp{k})$ using the recursion in Theorem~\ref{constsup}. As part of this process we divide the resulting partitions of~$k+1$ into the classes of Theorem~\ref{constsup} to prepare for the next step. We do this for $k_0\leq k <n$, taking $\sum_{k=k_0}^{n-1} \Theta(k^2)$, i.e., $\Theta(n^3)$ time units.
\par
We then iterate once over the~$\Theta(n^2)$ objects in $\mathcal{J}(\Lp{n})$ to get the attributes $\mathcal{J}(\Lp{n})^{*}$. Then for each object we iterate over these $\Theta(n^2)$ attributes to write the cells of the cross table, taking $\Theta(n^4)$ time units.\qed
\end{proof}
	
  The procedure for $\mathbb{K}(\Lp{n})$ described in the proof of Corollary~\ref{cor:std-context-algo} is near-optimal as already writing a non-trivial context of size $\Theta(n^2)\times\Theta(n^2)$ needs $\lo(n^4)$ individual steps.
\par
The standard context $\mathbb{K}(\Lp{n})$ is of our interest because its concept lattice $\mathfrak{B}(\mathbb{K}(\Lp{n}))$ is isomorphic to~$\Lp{n}$. From equation~\eqref{formula} we have that the number of objects (and attributes) of $\mathbb{K}(\Lp{n})$ has order $\Theta(n^2)$. This means that, as~$n$ increases, the number of objects (and attributes) in the context $\mathbb{K}(\Lp{n})$ grows much more slowly than the number~$p(n)$ of unrestricted partitions of~$n$, which satisfies the asymptotic formula~\cite{hardy1918}:
	\begin{align*}
	p(n)\sim \dfrac{1}{4n\sqrt{3}}e^{ \pi \sqrt{2n/3}}.
	\end{align*}
	
	It is quite remarkable (though from a formal concept analysis perspective not too surprising) that from a fraction of the partitions of~$n$, we can construct a context  $\mathbb{K}(\Lp{n})$ the size of which is `only' of order $\mathcal{O}(n^4)$ but whose number of formal concepts is precisely the total number of partitions of~$n$.

Asymptotic expansions~\cite{Rademacher1937,Rademacher1943} (and even exact numbers for many integers~$n$, see~\cite{OEISA000041}) for the sizes~$p(n)$ of the lattices $\mathfrak{B}(\mathbb{K}(\Lp{n}))\cong \Lp{n}$ are known. Moreover, by Corollary~\ref{cor:number-join-irr}, the standard contexts $\mathbb{K}(\Lp{n})$ satisfy precise size estimates $\mathcal{O}(n^4)$ and can be computed efficiently, that is, in polynomial time~$\Theta(n^4)$, see Corollary~\ref{cor:std-context-algo}. Therefore, in addition to being a novel, perhaps slightly esoteric, means for the computation of~$p(n)$, the sequence $\apply{\mathbb{K}(\Lp{n})}_{n\in\Nplus}$ of contexts might also prove itself to be a promising playground for testing conjectures or the efficiency of new formal concept analytic algorithms (regarding e.g.\ the computation of concept lattices or stem bases etc.).
%
%
%

\begin{thebibliography}{8}
\bibitem{Brylawski}
Brylawski, T.: The lattice of integer partitions. Discrete Math.
  \textbf{6}(3),  201--219 (1973).
  \doi{10.1016/0012-365X(73)90094-0}

\bibitem{davey}
Davey, B.A., Priestley, H.A.: Introduction to lattices and order. Cambridge
  University Press, New York, second edn. (2002).
  \doi{10.1017/CBO9780511809088}

\bibitem{Ganter20}
Ganter, B.: Notes on integer partitions. In 15th International Conference on Concept Lattices
  and Their Applications, pp. 19--31. CEUR-WS.org, Tallinn (2020)

\bibitem{ganter1997applied}
Ganter, B., Wille, R.: Applied lattice theory: {F}ormal concept analysis. In:
  Gr{\"a}tzer, G. (ed.) General Lattice Theory. Birkh{\"a}user, Basel, second
  edn. (1998)

\bibitem{ganter}
Ganter, B., Wille, R.: Formal concept analysis. {M}athematical foundations.
  Springer, Berlin (1999)

\bibitem{hardy1918}
Hardy, G.H., Ramanujan, S.: Asymptotic formul{\ae} in combinatory analysis.
  Proc. London Math. Soc. (2)  \textbf{17}(1),  75--115 (1918).
  \doi{10.1112/plms/s2-17.1.75}

\bibitem{LATAPY}
Latapy, M., Phan, T.H.D.: The lattice of integer partitions and its infinite
  extension. Discrete Math.  \textbf{309}(6),  1357--1367 (Apr 2009).
  \doi{10.1016/j.disc.2008.02.002}

\bibitem{mahnke1912leibniz}
Mahnke, D.: {L}eib\-niz auf der {S}u\-che nach einer all\-ge\-mei\-nen
  {P}rim\-zahl\-glei\-chung. Bibl. Math. (3)  \textbf{XIII},  29--61
  (1912--1913), \url{https://www.ophen.org/pub-102519}

\bibitem{OEISA000041}
{OEIS}: Sequence {A000041}. In: Sloane, N.J.A. (ed.) The on-line encyclopedia
  of integer sequences. OEIS Foundation (2020), \url{https://oeis.org/A000041},
  last accessed: 13 Dec 2020
%
\bibitem{Rademacher1937}
Rademacher, H.: On the partition function $p(n)$. Proc. London Math. Soc. (2)
  \textbf{43}(4),  241--254 (1937). \doi{10.1112/plms/s2-43.4.241}

\bibitem{Rademacher1943}
Rademacher, H.: On the expansion of the partition function in a series. Ann. of
  Math. (2)  \textbf{44}(3),  416--422 (Jul 1943). \doi{10.2307/1968973}
  
\end{thebibliography}
%

\end{document}